\documentclass[11pt]{article}
\pdfoutput=1
\usepackage{graphicx}          % Include this line if your 
\usepackage{url}                               % document contains figures,
\usepackage{natbib}            % you should have natbib.sty
\usepackage{amsthm}
\usepackage[ruled,vlined]{algorithm2e} 
\RequirePackage{xspace} \RequirePackage{amsmath, amssymb, textcomp,color}
\DeclareMathOperator*{\trace}{Tr}

\DeclareMathOperator*{\argmin}{argmin}
 \newcommand{\Tr}{\text{Tr}}
\usepackage{pdfsync}
\renewcommand{\Re}{{\mathbb{R}}}
\newcommand{\Ce}{\mathbb{C}} 
\usepackage{amssymb}

 \def\xx{{\boldsymbol{x}}}
 \def\B{{B}}
 \def\zz{{\boldsymbol{z}}} 
 \def\yy{{\boldsymbol{y}}}
\def\bb{{\boldsymbol{b}}}  
\def\aa{{\boldsymbol{a}}}  
 \DeclareMathOperator*{\rank}{rank}
\def\ee{{\boldsymbol{e}}}  
\def\subjto{{\mbox{subj. to}}}
\def\find{{\mbox{find }}}

\newcommand{\Rmnum}[1]{\MakeUppercase{\romannumeral #1}}
 
\renewcommand{\Re}{{\mathbb{R}}}

 \newtheorem{thm}{Theorem}

\newtheorem{df}[thm]{Definition}
\newtheorem{cor}[thm]{Corollary}

 \newtheorem{ex}[thm]{Example}
\usepackage{authblk}

\author[$\dagger$,*]{Henrik~Ohlsson} 
\author[*]{Allen~Y.~Yang} 
\author[*]{Roy~Dong} 
\author[*]{S.~Shankar~Sastry} 
\affil[$\dagger$]{Division of
  Automatic Control, Department of Electrical Engineering, Link\"oping
  University, Sweden.}        
\affil[*]{Department of Electrical Engineering and Computer
  Sciences, University of California at Berkeley, CA, USA\\  \{ohlsson,yang,roydong,sastry\}@eecs.berkeley.edu.}

\begin{document}

%\begin{frontmatter}
\title{Compressive Phase Retrieval From Squared Output Measurements Via Semidefinite Programming} %\thanks{Department of Electrical Engineering and Computer
%  Sciences, University of California at Berkeley, CA, USA 
 % \{ohlsson,yang,roydong,sastry\}@eecs.berkeley.edu} }
\maketitle

\begin{abstract}                         
Given a linear system in a real or complex domain, linear regression aims to recover the model parameters from a set of observations.
Recent studies in compressive sensing have successfully shown that under certain conditions, a linear program, namely, $\ell_1$-minimization, guarantees recovery of
sparse parameter signals even when the system is underdetermined. In this paper, we consider a more challenging problem: when
the phase of the output measurements from a linear system is omitted. 
Using a lifting technique, we show that even though the phase information is missing, the sparse signal can be
recovered exactly by solving a simple semidefinite program when the sampling rate is sufficiently
high, albeit the exact solutions to both sparse signal recovery and phase retrieval are combinatorial. 
The results extend the type of applications that compressive sensing can be applied to those where only output magnitudes can be observed.
We demonstrate the accuracy of the algorithms through theoretical analysis, extensive simulations and a practical experiment.
\end{abstract}

%\end{frontmatter}

%%%%%%%%%%%%%%%%%%%%%%%%%%%%%%%%%%%%%%%%%%%%%%%%%%%%%%%%%%%%%%%%%%%%%%%%%%%%%%%%%%%%%%%%%%%%%%%%%%%%%%%%%%%%%%%%%%%%%%%%%%%%%%%%%%%%%%%%%%%%%%%%%%%%%%%%%%%%%%%%%%%%%%%%%%%%%%%%%%%%%%%%%%%%%%%%%%%%%%%%%%%%%%%%%%%%%%%%%%%%%%%%%%%%%%%%%%%%%%%%%%%%%%%%%

\section{Introduction}
Linear models, e.g. $\yy = A\xx$, are by far the most used and useful type of model. The
main reasons for this are their simplicity of use and identification. For the
identification, the least-squares (LS) estimate in a complex domain is computed by\footnote{Our derivation in this paper is primarily focused on complex signals, but the results should be easily extended to real domain signals.}
\begin{equation}\label{eq:ls}
\xx_{\text{ls}}=\argmin_\xx \|\yy - A\xx\|_2^2 \in \Ce^n,
\end{equation}
assuming the output $\yy \in \Ce^N$ and $A \in \Ce^{N\times n}$ are given.
Further, the LS problem has a unique solution if the system is full
rank and not underdetermined, i.e. $ N\geq n$.

Consider the alternative scenario when the system is underdetermined, i.e. $n>N$. The least squares solution is no longer unique in this case,
and additional knowledge has to be used to determine a unique model parameter.
Ridge regression or Tikhonov regression \citep{Hoerl:70}
is one of the traditional methods to apply in this case, which takes the form
\begin{equation}   
 \xx_{\text{r}}=\argmin_\xx  \frac{1}{2}\|\yy - A\xx\|^2_2 +\lambda \|\xx\|^2_2,
\end{equation}
where $\lambda>0$ is a scalar parameter that decides the trade off between fit in the first term and the $\ell_2$-norm of $\xx$ in the second term.

Thanks to the $\ell_2$-norm regularization, ridge regression is known to pick up solutions with small energy that satisfy the linear model.
In a more recent approach stemming from the LASSO \citep{Tibsharami:96} and compressive sensing (CS) \citep{Candes:06,Donoho:06},
another convex regularization criterion has been widely used to seek the \emph{sparsest} parameter vector, which takes the
form
\begin{equation}\label{eq:l1}
 \xx_{\ell_1} =\argmin_\xx  \frac{1}{2}\|\yy - A\xx\|^2_2 +\lambda \|\xx\|_1.
\end{equation}
Depending on the choice of the weight parameter $\lambda$, the program \eqref{eq:l1}
has been known as the LASSO by \cite{Tibsharami:96}, \emph{basis pursuit denoising} (BPDN) by \cite{Chen:98}, or $\ell_1$-minimization ($\ell_1$-min) by \cite{Candes:06}.
In recent years, several pioneering works have contributed to efficiently solving sparsity minimization problems
such as \citep{Tropp:04,BeckA2009,bruckstein:09}, especially when the system parameters and observations are in high-dimensional spaces.

% In this paper, we consider a more challenging problem than in an underdetermined system, only
% the squared output measurements are given
% \begin{equation}
% b_i = y_i^2 = | \langle \xx, \aa_i\rangle |^2, \quad i= 1, \dots, N,
% \end{equation}
% where $A^T = [\aa_1, \cdots, \aa_n] \in \Ce^{N\times n}$,
% $N<n$. \TODO{Candes' scheme finds unique solutions when $N>c_0 n \log
% n$, I would rather say that we handle that case}

In this paper, we consider a more challenging problem. In a linear model $\yy = A\xx$, rather than assuming that $\yy$ is given,
we will assume that only the squared magnitude of the output is observed:
\begin{equation}\label{eq:pr}
b_i =   |  y_i |^2= | \langle \xx, \aa_i\rangle |^2, \quad i= 1, \cdots, N,
\end{equation} where $A^H = [\aa_1, \cdots, \aa_N] \in \Ce^{n\times
  N}$, $\yy^T = [y_1, \cdots, y_N] \in \Ce^{1 \times
  N}$ and $A^H$ denotes  the Hermitian transpose of $A$.
This is clearly a more challenging problem since the phase of $\yy$ is
lost when only its (squared) magnitude is available. 
A classical example is that $\yy$ represents the Fourier transform of $\xx$, 
and that  only the Fourier transform modulus is observable. 
%the real measurements only collect the squared magnitude of the transform resul%t.
This scenario arises naturally in several practical applications such
as optics \citep{WaltherA1963,MillaneR1990}, coherent diffraction
imaging \citep{FienupJ1987}, and astronomical imaging
\citep{DaintyJ1987} and is known as the \emph{phase retrieval} problem.

%\TODO{For a Fourier basis, also translation, reflection, etc, but I
 % dont think we need to care. Let us just discuss the general case.}
We note that in general phase  cannot be uniquely
recovered regardless whether the linear model is overdetermined or not. A
simple example to see this, is if $\xx_0\in\Ce^n$ is a solution to $\yy =
A\xx$, then for any scalar $c\in\Ce$ on the unit circle $c\xx_0$ leads
to the same squared output $\bb$. 
%In the special case where the
%dictionary $A$ represents the unitary discrete Fourier transform (DFT)   
As mentioned in \citep{Candes:11b},
when the dictionary $A$ represents the unitary discrete Fourier
transform (DFT), the ambiguities may represent time-reversed solutions
or time-shifted solutions of the ground truth signal $\xx_0$. These
global ambiguities caused by losing the phase information are
considered acceptable in phase retrieval applications. From
now on, when we talk about the solution to the phase retrieval problem, it
is the solution up to a global phase ambiguity. Accordingly, a unique solution
is a solution unique up to a global phase.

Further note that since \eqref{eq:pr} is nonlinear in the unknown $\xx$,
$N \gg n$ measurements are in general needed for a unique
solution. When the number of measurements $N$ are fewer than necessary
for a unique solution, additional assumptions are needed to
select one of the solutions (just like in Tikhonov, Lasso and CS).

Finally, we note that the exact solution to either CS and phase retrieval is combinatorially expensive \citep{Chen:98,Candes:11}.
Therefore, the goal of this work is to answer the following question: \emph{Can we effectively recover 
a sparse parameter vector $\xx$ of a linear system up to a
global ambiguity using its squared magnitude output measurements via convex
programming?} The problem is referred as \emph{compressive phase retrieval} (CPR) \citep{MoravecM2007}.

The main contribution of the paper is a \emph{convex formulation} of the
sparse phase retrieval problem. Using a lifting technique, the NP-hard
problem is relaxed as a semidefinite program. We also derive bounds for
guaranteed recovery of the true signal and compare the performance of our CPR algorithm with
traditional CS %\citep{MoravecM2007} 
and PhaseLift \citep{Candes:11b}
algorithms through extensive experiments.
The results extend the type of applications that compressive sensing can be applied to; namely, applications
where only magnitudes can be observed.

\subsection{Background}
Our work is motivated by the $\ell_1$-min problem in CS and a recent PhaseLift technique in phase retrieval by \cite{Candes:11}. On one hand, the theory of CS and $\ell_1$-min has 
been one of the most visible research topics in recent years. There
are several comprehensive review papers that cover the literature of
CS and related optimization techniques in linear programming. The
reader is referred to the works of
\citep{candes:08,bruckstein:09,LorisI2009,YangA2010-ICIP}. On the other
hand, the fusion of phase retrieval and matrix completion is a novel
topic that has recently been studied in a selected few papers, such
as \citep{Candes:11,Candes:11b}. The fusion of phase retrieval and CS
was discussed in \citep{MoravecM2007}. In the rest of the section, we briefly review the phase retrieval literature and its recent connections with CS and matrix completion.

Phase retrieval has been a longstanding problem in optics and x-ray
crystallography since the 1970s \citep{KohlerD1972,GonsalvesR1976}. Early methods to recover the phase signal using Fourier transform mostly relied on additional information about the signal, such as band limitation, nonzero support, real-valuedness, and nonnegativity. The Gerchberg-Saxton algorithm was one of the popular algorithms that alternates between the Fourier and inverse Fourier transforms to obtain the phase estimate iteratively \citep{GerchbergR1972,FienupJ1982}. One can also utilize steepest-descent methods to minimize the squared estimation error in the Fourier domain \citep{FienupJ1982,MarchesiniS2007}. Common drawbacks of these iterative methods are that they may not converge to the global solution, and the rate of convergence is often slow. Alternatively, \cite{BalanR2006} have studied a frame-theoretical approach to phase retrieval, which necessarily relied on some special types of measurements.

More recently, phase retrieval has been framed as a low-rank matrix
completion problem in \citep{Chai:10,Candes:11b,Candes:11}. Given a system, a
lifting technique was used to approximate the linear model constraint
as a semidefinite program (SDP), which is similar to the objective function of the proposed method only without the sparsity constraint. The authors also derived the upper-bound for the sampling rate that guarantees exact recovery in the noise-free case and stable recovery in the noisy case.

We are aware of the work by \cite{MoravecM2007}, which has considered
compressive phase retrieval on a random Fourier transform model. 
Leveraging the sparsity constraint, the authors proved
that an upper-bound of $O(k^2\log(4n/k^2))$  random Fourier modulus
measurements to uniquely specify $k$-sparse
signals. \cite{MoravecM2007} also proposed a greedy compressive  phase
retrieval algorithm. Their solution largely follows the development of $\ell_1$-min in CS, and it
alternates between the domain of solutions that give rise to the same
squared output and the domain of an $\ell_1$-ball with a fixed
$\ell_1$-norm. However, the main limitation of the algorithm is
that it tries to solve a nonconvex optimization problem and that it assumes
the $\ell_1$-norm of the true signal is known. No guarantees for when
the algorithm recovers the true signal can therefore be given.

\section{CPR via SDP}
%When the number
%$N$ of squared magnitude output measurements is larger  than the
%number $n$ of unknowns, 
In the noise free case, the phase retrieval problem takes the form of the feasibility
problem:
\begin{equation}
\find \xx\quad \subjto\quad \bb = |A\xx|^2 = \{\aa_i^H \xx \xx^H \aa_i \}_{1\le i\le N},
\label{eq:phase-retrieval}
\end{equation}
where $\bb^T = [b_1, \cdots, b_N] \in \Re^{1 \times
  N}$. 
%In the noise free case, it suffices to pick out $n$ of the $N$ \TODO{CORRECT???}
%squared magnitude observations, assuming the corresponding elements in
%the dictionary $A$ are linearly independent and uniformly distributed
%on the unit sphere of $\Ce^n$ or $\Re^n$. 
This is a
combinatorial problem to solve: Even in the real domain with the sign
of the measurements $\{\alpha_i\}_{i=1}^N\subset \{-1, 1 \}$, one
would have to try out combinations of sign sequences until one that satisfies
\begin{equation}
\alpha_i \sqrt{b_i} = \aa_i^T\xx,\quad i=1,\cdots, N,
\end{equation}
for some $\xx\in \Re^n$ has been found. For any practical size of
data sets, this combinatorial problem is intractable.
%As we previously mentioned, in the case of CPR, we assume 

% The feasibility problem \eqref{eq:phase-retrieval}, is a nonlinear
% equation system and the notion of underdetermined and overdetermined
% equation systems used for linear equation systems does not apply. In fact, it has been
% showed that when  $\aa_i,i=1,\cdots,N$ is taken as the Fourier basis
% then in general $2n$ measurements are needed to guarantee a unique
% solution \cite{??}.  
% However, if the signal $\xx$ is a priori known to be sparse, considerably
% fewer measurements are needed \cite{??}.

%All signals that people find meaningful can be decomposed as a sparse
%basis function expansion  \cite{Hayes:2009}. A sequence of
%independent random %numbers is an example of a signal that can not be
%decomposed using a sparse %decomposition.

Since \eqref{eq:phase-retrieval} is nonlinear in the unknown $\xx$,
$N \gg n$ measurements are in general needed for a unique
solution. When the number of measurements $N$ are fewer than necessary
for a unique solution, additional assumptions are needed to
select one of the solutions. Motivated by compressive
sensing, we here choose to seek the sparsest solution of CPR satisfying
\eqref{eq:phase-retrieval} or, equivalent, the solution to 
%If the signal $\xx$ is known to be sparse we could aim to solve the following %problem
\begin{equation}
\min_{\xx} \|\xx\|_0, \quad \subjto\quad \bb = |A\xx|^2 = \{\aa_i^H \xx \xx^H \aa_i \}_{1\le i\le N}.
\label{eq:CPR-exact}
\end{equation}
As the counting norm $\|\cdot\|_0$ is not a convex function, following the $\ell_1$-norm relaxation in CS, \eqref{eq:CPR-exact} can be relaxed as
\begin{equation}
\min_{\xx} \|\xx\|_1, \quad \subjto\quad \bb = |A\xx|^2 = \{\aa_i^H \xx \xx^H \aa_i \}_{1\le i\le N}.
\label{eq:CPR-l1}
\end{equation}

Note that \eqref{eq:CPR-l1} is still not a linear program, as its equality constraint is not a linear equation.
In the literature, a lifting technique has been extensively used to reframe problems such as \eqref{eq:CPR-l1} to a standard form in semidefinite programming, such as in Sparse PCA \citep{Aspremont:07}. 

More specifically, given the ground truth signal $\xx_0\in\Ce^n$, let
$X_0\doteq \xx_0\xx_0^H\in \Ce^{n\times n}$ be an induced rank-1
semidefinite matrix. Then the compressive phase retrieval (CPR) problem can be cast as\footnote{In this paper, $\|X\|_1$ for a matrix $X$ denotes the entry-wise $\ell_1$-norm, and $\|X\|_2$ denotes the  Frobenius norm.}
\begin{eqnarray}
\begin{array}{rl}
\min_X& \|X\|_1\\
\subjto& b_i = \trace(\aa_i^HX\aa_i),\; i=1,\cdots,N, \\
& \rank(X) = 1, X \succeq 0.
\end{array}
\end{eqnarray}

This is of course still a non-convex problem due to the rank
constraint. The lifting approach addresses this issue by replacing
$\rank(X)$ with $\trace(X)$. For a semidefinite matrix, $\trace(X)$ is
equal to the sum of the eigenvalues of $X$ (or the $\ell_1$-norm on
a vector containing all eigenvalues of $X$). This leads to an SDP
\begin{eqnarray}
\begin{array}{rl}
\min_X& \trace(X)+ \lambda \|X\|_1\\
\subjto& b_i = \trace(\Phi_iX), \; i=1,\cdots,N, \\
& X \succeq 0,
\end{array}
\label{eq:noiseless-SDP}
\end{eqnarray}
where we further denote $\Phi_i\doteq \aa_i\aa_i^H\in\Ce^{n\times
  n}$  and where $\lambda > 0$ is a design parameter.  Finally, the
estimate of $\xx$ can be found by computing the rank-1 decomposition of $X$
via singular value decomposition. We will refere to the formulation \eqref{eq:noiseless-SDP} as
 compressive phase
retrieval via lifting (CPRL).

We compare \eqref{eq:noiseless-SDP} to a recent solution of PhaseLift by \cite{Chai:10,Candes:11}.
In \cite{Chai:10,Candes:11}, a similar objective function was employed for phase retrieval:
\begin{eqnarray}
\begin{array}{rl}
\min_X& \trace(X)\\
\subjto& b_i = \trace(\Phi_iX), \; i=1,\cdots,N, \\
& X \succeq 0,
\end{array}
\label{eq:phase-lift}
\end{eqnarray}
albeit the source signal was not assumed sparse.
Using the lifting technique to construct the SDP relaxation of the NP-hard phase retrieval problem, with high probability, 
the program \eqref{eq:phase-lift} recovers the exact solution (sparse or
dense) if the number of measurements $N$ is at least of the order of $O(n \log n)$. The region of success
is visualized in Figure \ref{fig:nNrel} as region {\Rmnum 1}. 

If $\xx$ is sufficiently sparse and random Fourier
 dictionaries are used for sampling, \cite{MoravecM2007} showed that 
 in general the signal is uniquely defined if the number of
 squared magnitude output
 measurements $\bb$ exceeds the order of $O(k^2\log(4 n/k^2))$. This lower bound for the region of success of CPR is illustrated by the dash line in Figure \ref{fig:nNrel}.

Finally, the motivation for introducing the $\ell_1$-norm regularization in \eqref{eq:noiseless-SDP} is to be able to solve the sparse phase retrieval
problem for $N$ smaller than what PhaseLift requires. However, one
will not be able to solve the compressive phase retrieval problem in
region {\Rmnum 3} below the dashed
curve. Therefore, our target problems lie in region {\Rmnum 2}.

\begin{figure}[th!]
\centering
\includegraphics[width = 0.8\columnwidth]{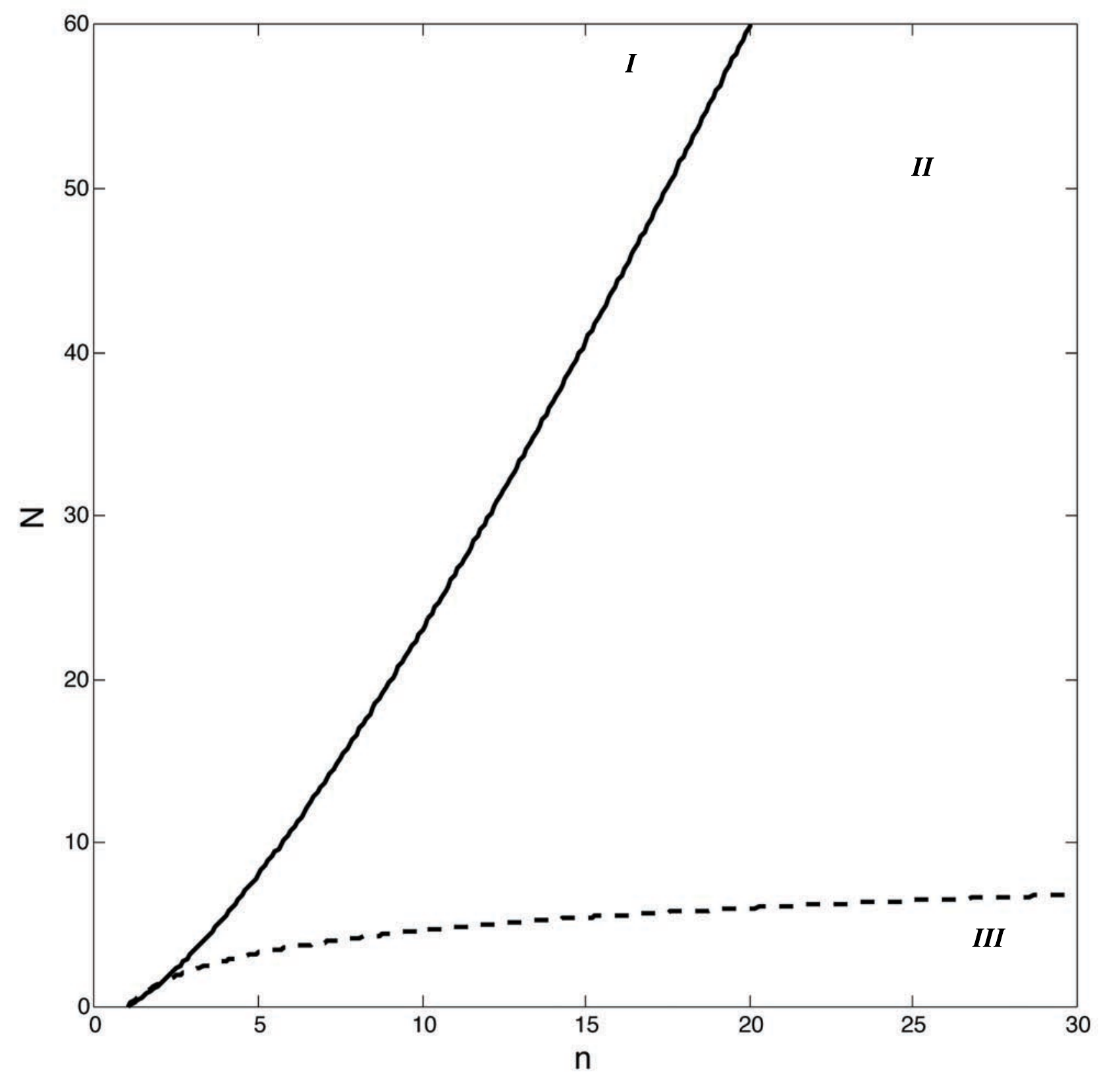}
\caption{An illustration of the regions of importance in solving the
  phase retrieval problem.
 While PhaseLift primarily targets problems in region {\Rmnum 1}, CPRL operates primarily in regions {\Rmnum 2} and {\Rmnum 3}.}
\label{fig:nNrel}
\end{figure}

\begin{ex}[Compressive Phase Retrieval]\label{ex:ex1}
In this example, we illustrate a simple CPR example,
where a 2-sparse complex signal $\xx_0\in\Ce^{64}$ is first
transformed by the Fourier transform $F\in \Ce^{64 \times 64}$
followed by random projections $R \in \Ce^{32\times 64}$ (generated by
sampling a unit complex Gaussian):
\begin{equation}
\bb=| R F \xx_0|^2.
\end{equation}

Given $\bb$, $F$, and $R$, we first apply PhaseLift  algorithm \cite{Candes:11} with $A=RF$ to
the $32$ squared observations $\bb$. The recovered dense signal is
shown in Figure~\ref{fig:firstex}. As seen in the figure, PhaseLift fails to identify the 2-sparse signal.

Next, we apply  CPRL \eqref{eq:noisy-SDP}, and the recovered sparse signal is also shown in
Figure~\ref{fig:firstex}. CPRL correctly identifies the two nonzero
elements in $\xx$.

\begin{figure}[th!]
\centering
\includegraphics[width = 0.8\columnwidth]{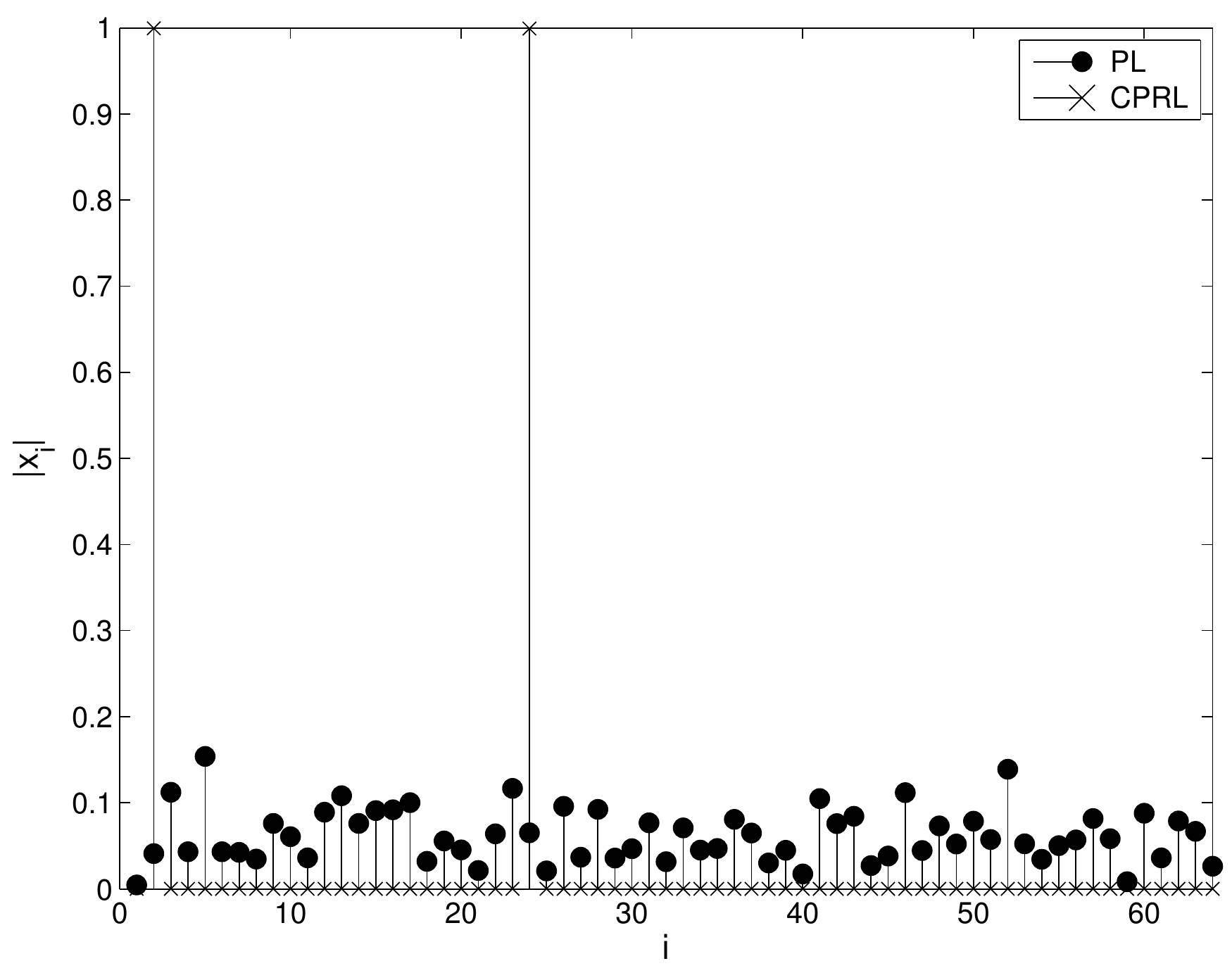}
\caption{The magnitude of the estimated signal provided by CPRL and PhaseLift
  (PL). CPRL correctly identifies elements 2 and 24 to be nonzero while
  PhaseLift provides a dense estimate. It is also verified that the estimate from CPRL, after a global phase shift,
  is approximately equal the true $\xx_0$.}
\label{fig:firstex}
\end{figure}

\end{ex}

\section{Stable Numerical Solutions for Noisy Data}
\label{sec:num}
In this section, we consider the case that the measurements are contaminated by data noise. In a linear model, typically bounded random noise affects the output of the system as $\yy = A\xx + \ee$,
where $\ee\in\Ce^N$ is a noise term with bounded $\ell_2$-norm: $\|\ee\|_2\le \epsilon$. However, in phase retrieval, we follow closely a more special noise model used in \cite{Candes:11}:
\begin{equation}
b_i = |\langle \xx, \aa_i\rangle|^2 + e_i.
\label{eq:noise-model}
\end{equation}
This nonstandard model avoids the need to calculate the squared
magnitude output
$|\yy|^2$ with the added noise term. More importantly, in practical phase
retrieval applications, measurement noise is introduced when the
squared magnitudes or intensities of the linear system are measured,
not on $\yy$ itself (\cite{Candes:11}).

Accordingly, we denote a linear operator $\B$ of $X$ as
\begin{equation}
\B: X\in \Ce^{n\times n} \mapsto \{\trace(\Phi_i X) \}_{1\le i \le N}\in\Re^{N},
\end{equation}
which measures the noise-free squared output. Then the approximate CPR problem with bounded  $\ell_2$ error model \eqref{eq:noise-model} can be solved by the following SDP program:
\begin{equation}
\begin{array}{rl}
\min& \trace(X)+ \lambda \|X\|_1\\
\subjto& \|\B(X) - \bb \|_2 \le \epsilon, \\
& X \succeq 0.
\end{array}
\label{eq:noisy-SDP}
\end{equation}
 The estimate of $\xx$, just as in noise free case, can finally be
 found by computing the rank-1 decomposition of $X$ via singular value
 decomposition. We refer to the method as approximate CPRL.
Due to the machine rounding error, in general a nonzero $\epsilon$
should be always assumed in the objective \eqref{eq:noisy-SDP} and its
termination condition during the optimization.

We should further discuss several numerical issues in the implementation of the SDP program. The constrained CPRL formulation \eqref{eq:noisy-SDP} can be rewritten as an unconstrained objective function:
\begin{equation}
\min_{X\succeq 0} \trace(X)+ \lambda \|X\|_1 + \frac{\mu}{2}\|\B(X) - \bb \|_2^2,
\label{eq:unconstrained}
\end{equation}
where $\lambda>0$ and $\mu>0$ are two penalty parameters.

In \eqref{eq:unconstrained}, due to the lifting process, the rank-1
condition of $X$ is approximated by its trace function $Tr(X)$. 
%Therefore, first, we consider under which condition the optimal
%solution of $X$ retains the rank-1 condition as an outerproduct $X
%=\xx \xx^H$. 
In \cite{Candes:11}, the authors considered phase retrieval of generic
(dense) signal $\xx$. They proved that if the number of measurements
obeys $N\ge c n\log n$ for a sufficiently large constant $c$, with
high probability, minimizing \eqref{eq:unconstrained} without the
sparsity constraint (i.e. $\lambda=0$) recovers a unique rank-1
solution obeying $X^* = \xx \xx^H$.  See also \citet{Recht:10}.

In Section \ref{sec:experiment}, we will show that using either random Fourier dictionaries or more general random projections, in practice, one needs much fewer measurements to exactly recover sparse signals if the measurements are noise free. Nevertheless, in the presence of noise, the recovered lifted matrix $X$ may not be exactly rank-1. In this case, one can simply use its rank-1 approximation corresponding to the largest singular value of $X$.

We also note that in \eqref{eq:unconstrained}, there are two main
parameters $\lambda$ and $\mu$ that can be defined by the
user. Typically $\mu$ is chosen depending on the level of noise that
affects the measurements $\bb$. For $\lambda$ associated with the
sparsity penalty $\|X\|_1$, one can adopt a \emph{warm start} strategy
to determine its value iteratively. The strategy has been widely used
in other sparse optimization, such as in $\ell_1$-min
\citep{YangA2010-ICIP}. More specifically, the objective is solved
iteratively with respect to a sequence of monotonically decreasing
$\lambda$, and each iteration is initialized using the optimization
results from the previous iteration. The procedure continues until a rank 1
solution has been found. When $\lambda$ is large, the sparsity constraint outweighs the trace constraint and the estimation error constraint, and vice versa.

\begin{ex}[Noisy Compressive Phase Retrieval]
Let us revisit Example~\ref{ex:ex1} but now assume that the
measurements are contaminated by noise. Using exactly the same data as
in Example~\ref{ex:ex1} but adding uniformly
distributed measurement noise between $-1$ and $1$, CPRL was able to
recover the 2  nonzero elements. PhaseLift, just as in
Example~\ref{ex:ex1} gave a dense estimate of $\xx$.
\end{ex}
%%%%%%%%%%%%%%%%%%%%%%%%%%%%%%%%%%%%%%%%%%%%%%%%%%%%%%%%%%%%%%%%%
%%%%%%%%%%%%%%%%%%%%%%%%%%%%%%%%%%%%%%%%%%%%%%%%%%%%%%%%%%%%%%%%%
%%%%%%%%%%%%%%%%%%%%%%%%%%%%%%%%%%%%%%%%%%%%%%%%%%%%%%%%%%%%%%%%%

\section{Computational Aspects}
In this section we discuss computational issus of the proposed SDP
formulation, algorithms for solving the SDP and to efficient
approximative solution algorithms.

\subsection{A Greedy Algorithm}

Since \eqref{eq:noiseless-SDP} is an SDP, it can be solved by standard software, such as CVX \citep{CVX1}.
However, it is well known that the standard toolboxes suffer when the dimension of $X$
is large. We therefore propose a greedy approximate algorithm tailored to
solve  \eqref{eq:noiseless-SDP}. 
%Under mild conditions, the algorithm GCPRL can accurately solve large
%problems of the form \eqref{eq:noiseless-SDP}. 
If the number of nonzero elements in $x$ is expected to be low, the
following algorithm may be suitable and less computationally heavy
compare to approaching the original SDP:
% \begin{enumerate}
% \item Set $\mathcal{Q}=\{1,\cdots,N\}$ and pick $\epsilon$ to be a
%   small positive real scalar.
% \item For $i=1,\cdots,N$, solve
% \begin{equation}\begin{aligned}
% \min_X \trace(X) &\\ \subjto  \; \bb=&B(X) \\ X \succeq & 0\\
% X(j,:)=&X(:,j)=0, \; j\in \{1,\cdots,i-1,i+1,\cdots, N\} \bigcap
% \mathcal{Q}.
% \end{aligned}\end{equation}
% Let $W_i$ denote the objective value and $X_k$ the obtained solution.
% \item Let $k$ be such that $W_k\leq W_i, i=1,\cdots,N$. 
% \item If $W_k<\epsilon$ stop and return $X_k$. Otherwise, set
%   $\mathcal{Q}$ to $\mathcal{Q} \backslash k$ and return to step 2.
% \end{enumerate}

\begin{algorithm}[H]
%\dontprintsemicolon
%\KwIn{$\{\aa_i,b_i\}_{i=1}^N$, $\gamma >0$ and  $\epsilon >0$.}
%\KwOut{$X_p$}
Set $\mathcal{I} = \emptyset$ and let $\gamma >0$, $\epsilon >0$.\\
\Repeat{$W_p < \epsilon$}{
\For{$k = 1,\cdots,N,$}{
Set $\mathcal{I}_k =
    \mathcal{I} \bigcup \{k\}$ and solve $X_k \doteq \arg\min_{X \succeq 0}  
 \trace (X^{(\mathcal{I}_k)}) + \gamma \sum_{i=1}^N (b_i - \trace (
  \aa_i^{(\mathcal{I}_k )} {\aa_i^{(\mathcal{I}_k)}}^H X^{(\mathcal{I}_k)}))^2.$\\
  Let $W_k$ denote the corresponding objective value.
}
 Let $p$ be such that $W_p  \leq W_k, k=1,\cdots,N$.  Set
    $\mathcal{I} = \mathcal{I}  \bigcup \{p\}$ and $X = X_p$. }
\caption{Greedy Compressive Phase Retrieval via Lifting (GCPRL)}
\end{algorithm}

\begin{ex}[GCPRL ability to solve the CPR problem]

 To demonstrate the effectiveness of  GCPRL let us consider a numerical
 example.  Let  the true $\xx_0 \in  \Ce^n$ be a $k$-sparse signal,
 let the nonzero elements be randomly chosen and
 their values randomly distributed on the complex unit circle. Let $A\in
 \Ce^{N \times n}$ be  generated by sampling from a complex unit Gaussian distribution.

If we fix $n/N=2$, that is, twice as many unknowns as measurements,
and apply GCPRL for different values of $n,\,N$ and $k$ we obtain the
computational times visualized in the left plot of  Figure~\ref{fig:exp}. In all
simulations $\gamma =10$ and $\epsilon =10^{-3} $ are used in
GCPRL. The true sparsity pattern was always recovered. Since GCPRL can
be executed in parallel, the simulation
times can be divided by the number of cores used (the average run time in Figure~\ref{fig:exp}
is computed on a standard laptop running Matlab, 2 cores,
and using CVX to solve the low dimensional SDP of GCPRL). 
The algorithm is several magnitudes faster than the standard interior-point methods used in CVX.

\begin{figure}[h!]
\centering
\includegraphics[height=0.7\columnwidth]{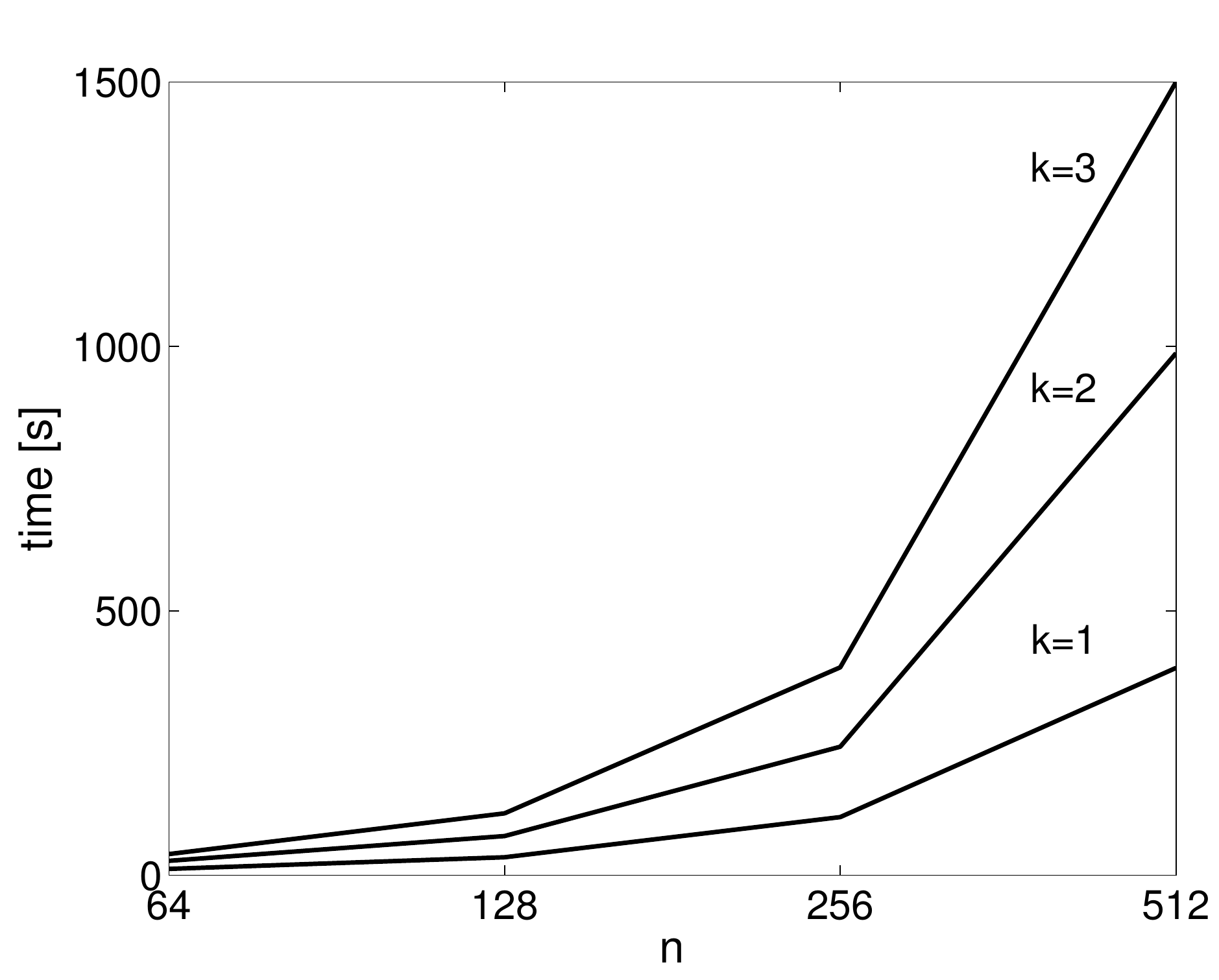}
\label{fig:exp}
\caption{Average run time of GCPRL in Matlab CVX environment. }
\end{figure}
\end{ex}

%%%%%%%%%%%%%%%%%%%%%%%%%%%%%%%%%%%%%%%%%%%%%%%%%%%%%%%%%%%%%%%%%
%%%%%%%%%%%%%%%%%%%%%%%%%%%%%%%%%%%%%%%%%%%%%%%%%%%%%%%%%%%%%%%%%
%%%%%%%%%%%%%%%%%%%%%%%%%%%%%%%%%%%%%%%%%%%%%%%%%%%%%%%%%%%%%%%%%

\section{The Dual}

CPRL takes the form:
\begin{eqnarray}
\begin{array}{rl}
\min_X  & \Tr(X) + \lambda \| X \|_1 \\
\subjto & b_i = \Tr( \Phi_i X); i = 1,\cdots,N, \\
& X \succeq 0.
\end{array}
\end{eqnarray}
If we define $h(X)$ to be an $N$-dimensional vector such that our constraints are $h(X) = 0$, then we can equivalently write:
\begin{eqnarray}
\begin{array}{rl}
\min_X \max_{\mu, Y, Z = Z^H} & \Tr(X) + \Tr(ZX) + \mu^T h(X) - \Tr(YX) \\
\subjto & Y \succeq 0, \\
& \|Z\|_{\infty} \leq \lambda.
\end{array}
\end{eqnarray}
Then the dual becomes:
\begin{eqnarray}
\begin{array}{rl}
max_{\mu, Z = Z^H} & \mu^T b \\
\subjto & \|Z\|_{\infty} \leq \lambda. \\
& Y := I + Z - \sum_{i=1}^N \mu_i \Phi_i \succeq 0.
\end{array}
\end{eqnarray}

%%%%%%%%%%%%%%%%%%%%%%%%%%%%%%%%%%
\section{Analysis}

This section contains various analysis results. The analysis follows
that of CS and have been inspired by derivations given in
\citep{Candes:11,Candes:06,Donoho:06,Candes_2008,berinde:08,bruckstein:09}. The analysis is
divided into four subsections. The three first subsections give results based on
RIP, RIP-1 and mutual coherency respectively. The last subsection
focuses on the use of Fourier dictionaries.

%%%%%%%%%%%%%%%%%%%%%%%%%%%%%%%%%%%%%%%%%%%%%%%%%%%%%%%%%%%%%%%%%
%%%%%%%%%%%%%%%%%%%%%%%%%%%%%%%%%%%%%%%%%%%%%%%%%%%%%%%%%%%%%%%%%
%%%%%%%%%%%%%%%%%%%%%%%%%%%%%%%%%%%%%%%%%%%%%%%%%%%%%%%%%%%%%%%%%

\subsection{Analysis Using RIP}
In order to state some theoretical properties, we need a generalization
 of the restricted isometry property (RIP). 
\begin{df}[RIP]\label{def:RIP}
We will say that a linear operator $\B(\cdot) $  is $(\epsilon,
k)$-RIP if for all $X\neq 0$ s.t. $\|X\|_0 \leq k$ we have 
\begin{equation}\label{eq:RIP}
\left | \frac{\| \B(X) \|_2^2}{\| X \|_2^2} -1\right|<\epsilon.
\end{equation}

\end{df}
We can now state the following theorem:

%%%%%%%%%%%%%%%%%%%%%%%%%%%%%%%%%%%%%%%%%%%%%%%%%%%%%%%%%%%%%%%%%

\begin{thm}[Recoverability/Uniqueness]\label{thm:rec}
%Let $\epsilon<1$ and let $\B(\cdot)$ be a $(\epsilon,2k)$-RIP linear operator. %Let $X^*$
%be a matrix s.t. $\|X^*\|_0\leq k$ and 
Let $\B(\cdot)$  be a $(\epsilon,2 \|X^*\|_0)$-RIP linear
operator with  $\epsilon<1$ and let $\bar \xx$ be the sparsest solution to \eqref{eq:pr}. 
If $X^*$ satisfies
\begin{equation}\label{eq:const}
\begin{aligned}
{\bf b}=& \B( X^*), \\ X^* \succeq &0,\\
\rank\{X^*\}=&1,
\end{aligned}
\end{equation}
then  $X^*$ is unique and  $X^*= \bar \xx \bar \xx^H$.
%is recovered by
%\begin{equation}\label{eq:CPR1}
%\begin{aligned}
%\min_X  \|X\|_0&\\
%\st \quad {\bf b}=& \B(X),
%\\ X \succeq &0,\\
%\rank\{X\}=&1.
%\end{aligned}\end{equation}
\end{thm}

%%%%%%%%%%%%%%%%%%%%%%%%%%%%%%%%%%%%%%%%%%%%%%%%%%%%%%%%%%%%%%%%%
%%%%%%%%%%%%%%%%%%%%%%%%%%%%%%%%%%%%%%%%%%%%%%%%%%%%%%%%%%%%%%%%%
%%%%%%%%%%%%%%%%%%%%%%%%%%%%%%%%%%%%%%%%%%%%%%%%%%%%%%%%%%%%%%%%%

\begin{proof}[Proof of Theorem \ref{thm:rec}]
Assume the contrary i.e., $X^* \neq \bar \xx \bar \xx^H$.
% that \eqref{eq:CPR1} gives a different
%solution than $X^*$, namely $\bar X$. 
It is clear that $\|\bar \xx \bar \xx^H \|_0 \leq \|X^*\|_0$ and
hence  $\|\bar \xx \bar \xx^H - X^* \|_0 \leq 2 \|X^*\|_0$. If we now apply the RIP
inequality \eqref{eq:RIP} on $\bar \xx \bar \xx^H -
X^*$ and use that $\B(\bar \xx \bar \xx^H -X^*)=0$ we are led to the contradiction
$1<\epsilon$. We therefore conclude that $X^*$ is unique and $X^*=\bar \xx \bar \xx^H$.
\end{proof} 
%%%%%%%%%%%%%%%%%%%%%%%%%%%%%%%%%%%%%%%%%%%%%%%%%%%%%%%%%%%%%%%%%
%%%%%%%%%%%%%%%%%%%%%%%%%%%%%%%%%%%%%%%%%%%%%%%%%%%%%%%%%%%%%%%%%
%%%%%%%%%%%%%%%%%%%%%%%%%%%%%%%%%%%%%%%%%%%%%%%%%%%%%%%%%%%%%%%%%
%Now, since the zero-norm make the optimization problem combinatorial
%we would rather replace the zero norm by the $\ell_1$-norm. We can state the
%following result: 
We can also give a bound on the sparsity of $\bar \xx$:
\begin{thm}[Bound on $\|\bar \xx \bar \xx^H\|_0$ from above]\label{thm:Phaseliftrel1}
Let $\bar \xx$ be the sparsest solution to \eqref{eq:pr}  and let $\tilde X $ be
the solution of  CPRL \eqref{eq:noiseless-SDP}.  If $\tilde X$ has
rank 1 then  $\| \tilde X\|_0\geq \| \bar \xx \bar \xx^H \|_0$.

\end{thm}

%%%%%%%%%%%%%%%%%%%%%%%%%%%%%%%%%%%%%%%%%%%%%%%%%%%%%%%%%%%%%%%%%
%%%%%%%%%%%%%%%%%%%%%%%%%%%%%%%%%%%%%%%%%%%%%%%%%%%%%%%%%%%%%%%%%
%%%%%%%%%%%%%%%%%%%%%%%%%%%%%%%%%%%%%%%%%%%%%%%%%%%%%%%%%%%%%%%%%

\begin{proof}[Proof of Theorem~\ref{thm:Phaseliftrel1}]
Let  $\tilde X $ be a rank-1  solution of  CPRL \eqref{eq:noiseless-SDP}.
By contradiction, assume $\|\tilde X\|_0  < \| \bar \xx \bar \xx^H \|_0$. Since $\tilde X$ satisfies the constraints of  \eqref{eq:pr}, it
 must give a lower objective value than $\bar \xx \bar \xx^H$ in
  \eqref{eq:pr} . This is a contradiction since $\bar \xx \bar \xx^H$ was assumed to be
 the solution of  \eqref{eq:pr}. Hence we must have that $\|\tilde
 X\|_0 \geq \|
 \bar \xx \bar \xx^H \|_0$.
\end{proof}
%%%%%%%%%%%%%%%%%%%%%%%%%%%

\begin{cor}[Guaranteed Recovery Using RIP]\label{thm:guartee1}
Let $\bar \xx$ be the sparsest solution to \eqref{eq:pr}. 
 The solution of CPRL  \eqref{eq:noiseless-SDP}, $\tilde X$, is equal to  $\bar
\xx \bar \xx^H$ if  it
has rank 1 and
$B $ is ($\epsilon, 2\|\tilde X\|_0$)-RIP with $\epsilon<1$.
\end{cor}
\begin{proof}[Proof of Corollary~\ref{thm:guartee1}]
This follows trivially from Theorem \ref{thm:rec} by realizing that
$\tilde X$ satisfy all properties of $X^*$.

% Assume the contrary i.e., that \eqref{eq:noiseless-SDP} gives a different
% solution than  $\bar \xx \bar \xx^H$. It is clear
% that $\|\bar \xx \bar \xx^H \|_0 \leq \| \tilde X\|_0$ and
% hence  $\|\tilde X - \bar \xx \bar \xx^H \|_0  \leq 2\| \tilde X\|_0$. If we now apply the RIP
% inequality \eqref{eq:RIP} on $\tilde X -
% \bar \xx \bar \xx^H$ and use that $\B(\tilde X -\bar \xx \bar \xx^H)=0$ we are led to the contradiction
% $1<\epsilon$. We therefore conclude that $\bar \xx \bar \xx^H=\tilde X$.
\end{proof}

%%%%%%%%%%%%%%%%%%

If $\bar \xx \bar \xx^H=\tilde X$ can not be guaranteed, the following
bound could come useful:
\begin{thm}[Bound on $\| X^* -\tilde X \|_2$]\label{thm:CPR}
Let $\epsilon < \frac{1}{1+\sqrt{2}}$ and assume $\B(\cdot)$ to be a $(\epsilon,
2k)$-RIP linear operator. Let $X^*$  be any matrix (sparse or dense)
satisfying 
\begin{equation}\label{eq:const}
\begin{aligned}
{\bf b}=& \B( X^*), \\ X^* \succeq &0,\\
\rank\{X^*\}=&1,
\end{aligned}
\end{equation}
let $\tilde X$ be the CPRL solution,  \eqref{eq:noiseless-SDP},
and form $X_s$ from $X^*$ by setting all but the $k$ largest elements to zero, i.e.,
\begin{equation}
X_s=\argmin_{X:\|X\|_0\leq k} \|X^*-X\|_1.
\end{equation}
Then,
\begin{align}
\| \tilde X-X^* \|_2\leq &\frac{2}{(1-\rho) \sqrt{k} } \|X^*
-X_s\|_1 \nonumber \\ +&(2(1-\rho)^{-1}  +k^{-1/2}) \frac{1}{\lambda} (\trace X^* -\trace \tilde
 X)
\end{align}
with $\rho = \sqrt{2} \epsilon/(1-\epsilon)$. 
\end{thm}
\begin{proof}[Proof of Theorem \ref{thm:CPR}]
The proof is inspired by the work on compressed sensing presented in
\cite{Candes_2008}.

First, we introduce $\Delta  = \tilde X-X^*$. For a matrix $X$ and an index set $T$, we use the notation $X_{T}$ to mean the matrix with all zeros
except those indexed by $T$, which are set to the corresponding
values of $X$.  Then let $T_0$ be the index set of the $k$ largest elements of $X^*$ in absolute value,
and $T_0^c=\{(1,1), (1,2),\dots,(n,n) \} \setminus T_0$ be its complement.
Let $T_1$ be the index set associated
with the $k$ largest elements in absolute value of $\Delta
_{T_0^c}$ and $T_{0,1}\doteq T_0\cup T_1$ be the union. Let $T_2$ be the index set associated
with the $k$ largest elements in absolute value of $\Delta
_{T_{0,1}^c}$, and so on.

Notice that
\begin{equation}
\|\Delta \|_2 = \| \Delta_{T_{0,1}} + \Delta_{T_{0,1}^c} \|_2 \leq \|
\Delta_{T_{0,1}} \|_2+  \| \Delta_{T_{0,1}^c} \|_2.
\end{equation}
We will now study each of the two terms on the right hand side
separately.

We first consider  $\| \Delta_{T_{0,1}^c} \|_2$. For $j>1$ we have
that for each $i\in T_j$ and $i' \in T_{j-1}$ that $|\Delta[i]|\leq
|\Delta [i']|$. Hence $\|\Delta_{T_{j}}\|_{\infty} \leq
\|\Delta_{T_{j-1}}\|_1/k$. Therefore,
\begin{equation}
\| \Delta_{T_j}\|_2 \leq k^{1/2}\|\Delta_{T_{j}}\|_{\infty}  \leq k^{-1/2} \|\Delta_{T_{j-1}}\|_1 
\end{equation}
and
\begin{equation}
\| \Delta_{T_{0,1}^c} \|_2 \leq \sum_{j\ge 2} \| \Delta_{T_{j}} \|_2 \leq
k^{-1/2} \| \Delta_{T_0^c}\|_1. 
\end{equation}
Now, since $\tilde X$ minimizes $ \trace X +\lambda \|X\|_1$, we have
\begin{equation}
\begin{array}{rcl}
 \trace X^* +\lambda \|X^*\|_1 &\geq&  \trace \tilde X +\lambda \| \tilde X\|_1\\
 &\geq&  \trace \tilde X +\lambda ( \|X^*_{T_0}\|_1 -\|\Delta_{T_0} \|_1
 \\ &+& \|
 \Delta_{T_0^c} \|_1 -\|X^*_{T_0^c}\|_1 ).
 \end{array}
\end{equation}
Hence,
\begin{equation}
 \| \Delta_{T_0^c} \|_1 \leq  \frac{-1}{\lambda} \trace \Delta -\|X^*_{T_0}\|_1 +\|\Delta_{T_0} \|_1   +\|X^*_{T_0^c}\|_1+\|X^*\|_1.
\end{equation}
Using the fact $\|X^*_{T_0^c}\|_1 = \|X^*-X_s\|_1=\|X^*\|_1 - \|X^*_{T_0}\|_1$, we
get a bound for $ \| \Delta_{T_0^c} \|_1$:
\begin{equation}
 \| \Delta_{T_0^c} \|_1 \leq  \frac{-1}{\lambda} \trace \Delta +\|\Delta_{T_0} \|_1   +2 \|X^*_{T_0^c}\|_1. 
\end{equation}
Subsequently, the bound for $\|\Delta_{T_{0,1}^c}\|_2$ is given by
\begin{align}
 \| \Delta_{T_{0,1}^c} \|_2 \leq  & k^{-1/2}( \frac{-1}{\lambda} \trace
 \Delta  +\|\Delta_{T_0} \|_1   +2 \|X^*_{T_0^c}\|_1 )  \\ \leq&  k^{-1/2}( \frac{-1}{\lambda} \trace\Delta   +2 \|X^*_{T_0^c}\|_1 ) +\|\Delta_{T_0} \|_2.
\end{align}

Next, we consider  $\| \Delta_{T_{0,1}} \|_2$. It can be shown by a
similar derivation as in \cite{Candes_2008} that
\begin{equation}
\| \Delta_{T_{0,1}} \|_2 \leq \frac{\rho}{1-\rho} k^{-1/2}
\|X^*-X_s\|_1-\frac{1}{\lambda}\frac{1}{1-\rho}\trace \Delta.
\end{equation}

Lastly, combine the bounds for $\| \Delta_{T_{0,1}^c} \|_2 $ and
$\| \Delta_{T_{0,1}} \|_2 $, and we get the final result:
\begin{align}
\|\Delta\|_2 \leq & \| \Delta_{T_{0,1}} \|_2 +  \| \Delta_{T_{0,1}^c}
\|_2 \\ \leq & -k^{-1/2} \frac{1}{\lambda} \trace \Delta   +2  k^{-1/2} \|X^*_{T_0^c}\|_1  +2\| \Delta_{T_{0,1}} \|_2
\\ \leq& - \left (2 \left (1-\rho \right)^{-1}+k^{-1/2} \right)
\frac{1}{\lambda}  \trace \Delta  \\ &   +2(1-\rho)^{-1} k^{-1/2} \|X^*-X_s\|_1.
\end{align}
\end{proof}

The bound given in Theorem~\ref{thm:CPR} is rather impractical since
it contains both $\|\tilde X - X^*\|_2$ and $\trace (\tilde X - X^*)$. The
weaker bound given in the following corollary does not have this problem:   
\begin{cor}[A Practical Bound on $\| \tilde X-X^*\|_2$]\label{lem:prac}
The bound on $\| \tilde X-X^*\|_2$ in Theorem \ref{thm:CPR}
can be relaxed to a weaker bound:
\begin{align}
\Big (1-\big (\frac{2 k^{1/2} }{1-\rho}+1 &\big) \frac{1}{\lambda} \Big)
 \| \tilde X-X^*\|_2  \\ \leq & \frac{2}{(1-\rho) \sqrt{k} } \|X^*
-X_s\|_1.
\end{align}
If $X^*$ is $k$-sparse, $\epsilon < \frac{1}{1+\sqrt{2}}$, and  $\B(\cdot)$ is an $(\epsilon,
2k)$-RIP linear operator, then we can guarantee that $\tilde X=X^*$ if
\begin{equation}
\lambda>\frac{2 k^{1/2} }{1-\rho}+1 
\end{equation}
and $\tilde X$ has rank 1.
\end{cor}

%%%%%%%%%%%%%%%%%%%%%%%%%%%%%%%%%%%%%%%
\begin{proof}[Proof of Corollary \ref{lem:prac}]
It follows from the assumptions of Theorem \ref{thm:CPR} that
\begin{equation}
1-\rho=1- \frac{\sqrt{2} \epsilon}{1-\epsilon} \geq 1- \frac{\sqrt{2}
  \frac{1}{1+\sqrt{2}}}{1- \frac{1}{1+\sqrt{2}}} =0.
\end{equation} 
Hence,
\begin{equation}
\left (2\left (1-\rho \right )^{-1}+k^{-1/2} \right) \frac{1}{\lambda} \geq 0.
\end{equation}
Therefore, we have
\begin{align}
\| \tilde X-X^*\|_2\leq & \frac{2}{(1-\rho) \sqrt{k} } \|X^*
-X_s\|_1 \\+ &\left (2 \left (1-\rho \right )^{-1}+k^{-1/2} \right ) \frac{1}{\lambda} (\trace X^* -\trace \tilde
 X) \\ \leq & \frac{2}{ \left (1-\rho \right ) \sqrt{k} } \|X^*
-X_s\|_1 \\+ & \left (2 \left (1-\rho \right )^{-1}+k^{-1/2} \right ) \frac{1}{\lambda} \|  X^* - \tilde
 X\|_1 \\ \leq & \frac{2}{(1-\rho) \sqrt{k} } \|X^*
-X_s\|_1 \\ + & \left (2 \left (1-\rho \right)^{-1}+k^{-1/2} \right ) \frac{k^{1/2}}{\lambda} \|  X^* -\tilde
 X\|_2
\end{align}
which is equal to the proposed condition after a rearrangement of the terms. 
\end{proof}

Given the above analysis, however, it may be the case that the linear operator $\B(\cdot)$ does not satisfy the RIP property
defined in Definition \ref{def:RIP}, as pointed out in \cite{Candes:11}.
Therefore, next we turn our attention to RIP-1 linear operators.
%\subsection{Analysis Using RIP-1 Assumptions}
%%%%%%%%%%%%%%%%%%%%%%%%%%%%%%%%%%%%%%%%%%%%%%%%%%%%%%%%%%%%%%%%%
%%%%%%%%%%%%%%%%%%%%%%%%%%%%%%%%%%%%%%%%%%%%%%%%%%%%%%%%%%%%%%%%%
%%%%%%%%%%%%%%%%%%%%%%%%%%%%%%%%%%%%%%%%%%%%%%%%%%%%%%%%%%%%%%%%%

\subsection{Analysis Using RIP-1}
We define RIP-1 as follows: 
\begin{df}[RIP-1]\label{def:RIP1}
A linear operator $\B(\cdot)$  is $(\epsilon,
k)$-RIP-1 if for all matrices $X \neq 0$ subject to $\|X\|_0 \leq k$,
we have 
\begin{equation}\label{eq:RIP1}
\left | \frac{\| \B(X)\|_1}{\| X\|_1} -1\right|<\epsilon.
\end{equation}
\end{df}

%%%%%%%%%%%%%%%%%%%%%%%%%%%%%%%%%%%%%%%%%%%%%%%%%%%%%%%%%%%%%%%%%
%%%%%%%%%%%%%%%%%%%%%%%%%%%%%%%%%%%%%%%%%%%%%%%%%%%%%%%%%%%%%%%%%
%%%%%%%%%%%%%%%%%%%%%%%%%%%%%%%%%%%%%%%%%%%%%%%%%%%%%%%%%%%%%%%%%

Theorems \ref{thm:rec}--\ref{thm:Phaseliftrel1} and
Corollary~\ref{thm:guartee1}  all hold
with RIP replaced by
RIP-1. The proofs follow those of the previous section with minor
modifications (basically replace the 2-norm with the
$\ell_1$-norm). The RIP-1 counterparts of Theorems \ref{thm:rec}--\ref{thm:Phaseliftrel1}  and
Corollary~\ref{thm:guartee1} are not restated in details
here. Instead we summarize the most  important property in the following theorem:  

%%%%%%%%%%%%%%%%%%%%%%%%%%%%%%%%%%%%%%%%%%%%%%%%%%%%%%%%%%%%%%%%%
%%%%%%%%%%%%%%%%%%%%%%%%%%%%%%%%%%%%%%%%%%%%%%%%%%%%%%%%%%%%%%%%%
%%%%%%%%%%%%%%%%%%%%%%%%%%%%%%%%%%%%%%%%%%%%%%%%%%%%%%%%%%%%%%%%%

\begin{thm}[Upper Bound \& Recoverability Through
  $\ell_1$]\label{thm:bound}

Let $\bar \xx$ be the sparsest solution to \eqref{eq:pr}. 
 The solution of CPRL  \eqref{eq:noiseless-SDP}, $\tilde X$, is equal to  $\bar
\xx \bar \xx^H$ if  it
has rank 1 and
$B(\cdot) $ is ($\epsilon, 2\|\tilde X\|_0$)-RIP-1 with $\epsilon<1$.

% Let  $\B(\cdot)$ be a  $(\epsilon, k)$-RIP-1, $\epsilon <1$, linear
% operator. Let $ \bar \xx \bar \xx^H$  be the sparsest solution to \eqref{eq:pr},
% let $\tilde X$ be the solution to \eqref{eq:noiseless-SDP}.
% and assume that it has rank 1 and is $\bar k$-sparse. Then 
% \begin{enumerate}
% \item $\|X^*\|_0 \leq \|\tilde X\|_0$;
% \item if  $k\geq 2 \bar k$, then $X^* =\tilde X$.
% \end{enumerate}

\end{thm}
\begin{proof}[Proof of Theorem~\ref{thm:bound}]
%The first property follows trivially from that $\tilde X$ satisfies all
%the constraints of \eqref{eq:CPR31}. 
The proof follows trivially for the proof Theorem \ref{thm:guartee1} (basically replace the 2-norm with the
$\ell_1$-norm). 

\end{proof}

\subsection{Analysis Using Mutual Coherence}
%%%%%%%%%%%%%%%%%%%%%%%%%%%%%

The RIP type of argument may be difficult to check for a given matrix and
are more useful for claiming results for classes of matrices/linear
operators. For instance, it has  been shown that random Gaussian matrices satisfy the RIP  with
high probability. However, given one sample of a random Gaussian matrix,
it is hard to check if it actually satisfies the RIP or not.  

Two alternative arguments are spark \cite{Chen:98} and mutual coherence \cite{Donoho:03b,CANDES:2009}.
The spark condition usually gives tighter bounds but is known to be difficult to
compute. On the other hand, mutual coherence may give less tight bounds,
but is more tractable. We will focus on mutual coherence here. 

Mutual coherence is defined as: 
\begin{df}[Mutual Coherence]
 For a matrix $A$, define  the \textit{mutual coherence}  as
\begin{equation}
\mu(A)=\max_{1\leq k,j \leq n, k\neq j} \frac{|\aa^H_k
  \aa_j|}{\|\aa_k\|_2\|\aa_j\|_2}.
\end{equation}
\end{df}

By an abuse of notation, let $\B$ be the matrix satisfying $\bb= \B X^s$ with $X^s$ being the vectorized version of $X$. We are now ready to state the following theorem:

\begin{thm}[Recovery Using Mutual Coherence]\label{thm:guartee2}
Let $\bar \xx$ be the sparsest solution to \eqref{eq:pr}. 
 The
solution of CPRL  \eqref{eq:noiseless-SDP}, $\tilde X$, is equal to  $\bar
\xx \bar \xx^H$ if  it
has rank 1 and
%\begin{equation} 
$\|\tilde X\|_0 < 0.5(1+1/\mu(\B )).$
%\end{equation} 
\end{thm}
\begin{proof}[Proof of Theorem~\ref{thm:guartee2}]
Since $\bar \xx \bar \xx^H$ satisfies $\bb=B(\bar
\xx \bar \xx^H)$, it follows from \cite[Thm.~1]{Donoho:03b}
that 
\begin{equation}\label{eq:conduniquenessco}
\| \bar
\xx \bar \xx^H\|_0 < \frac{1}{2} \left ( 1 + \frac{1}{\mu(\B )}\right)
\end{equation}
is a sufficient condition for $\bar
\xx \bar \xx^H$ to be a unique solution. It further follows that if $\tilde X$ also satisfies
\eqref{eq:conduniquenessco} then we must have that  $\tilde X=\bar
\xx \bar \xx^H$ since $\tilde X$ also satisfies  $\bb=B(\tilde X)$.
\end{proof}

\section{Experiment}
\label{sec:experiment}
This section gives a number of comparisons with the other state-of-the-art methods in compressive phase retrieval. Code for the numerical illustrations
can be downloaded from \url{http://www.rt.isy.liu.se/~ohlsson/code.html}.

\subsection{Simulation}
First, we repeat the simulation given in Example~\ref{ex:ex1} for
$k=1,\dots,5$. For each $k$, $n=64$ is fixed, and we increase the measurement dimension $N$ until CPRL recovered the true sparse support in at least 95 out of 100 trials, i.e., 95\% success rate.
%\footnote{To decide
%whether or not the true sparsity pattern had been recovered, the following test is carried out:
%\begin{equation}\label{eq:test}
%\{i:|\xx_{\text{CPR},i})|<0.1\} \bigcap\{i:|\xx_i|>0.1 \} = \O
%\end{equation}
% $\xx_{\text{CPR},i}$ denotes the $i$ element of the CPR
%estimate 
%$\xx_{\text{CPR}}$, and $\xx_i$ denotes the $i$ element of 
%$\xx$.}
New data ($\xx$, $\bb$, and $R$) are generated in each trial. The curve of 95\% success rate is shown in Figure~\ref{fig:nVSk}.
\begin{figure}[th!]
\centering
\includegraphics[width = 0.8\columnwidth]{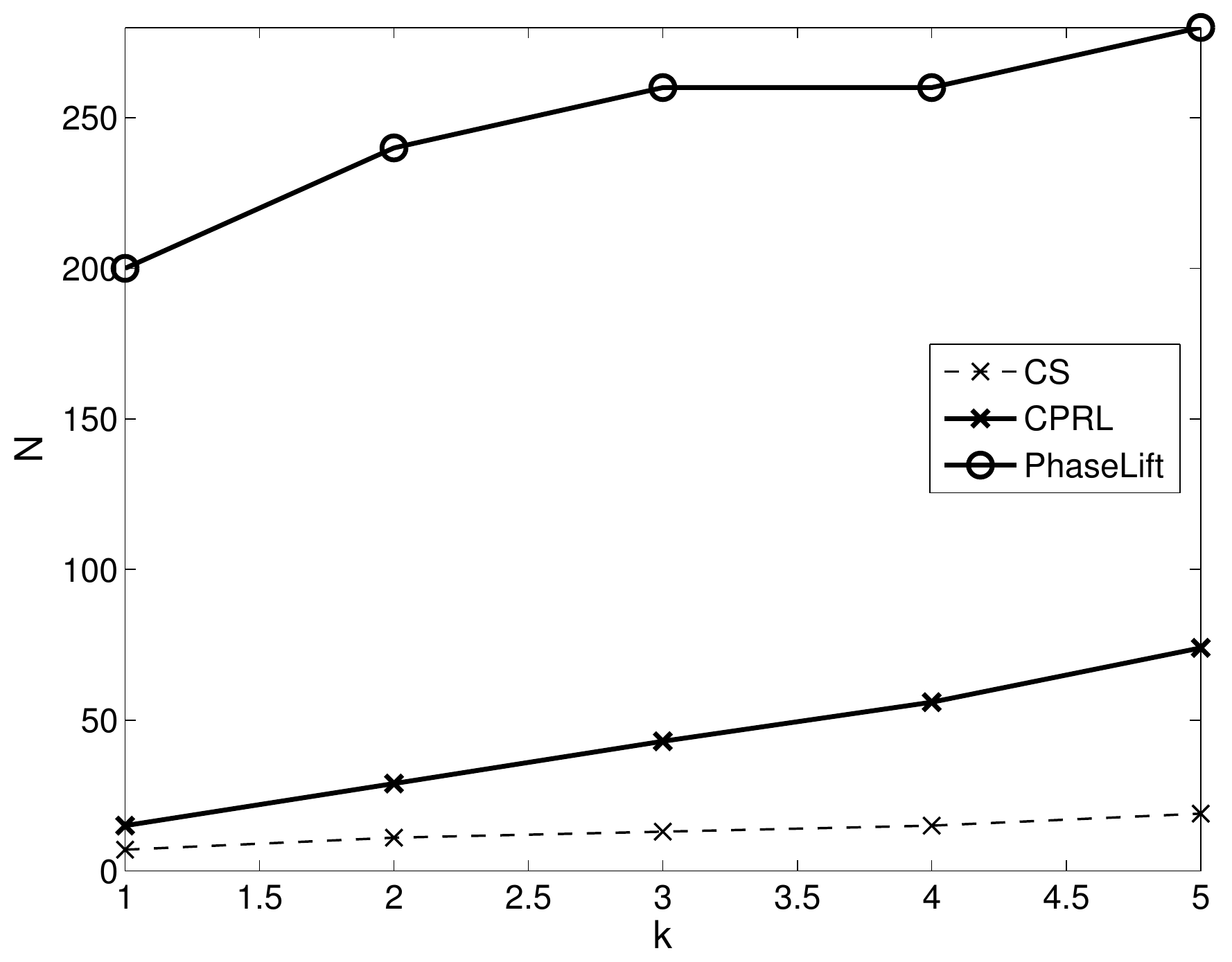}
\caption{The curves of 95\% success rates for CPRL, PhaseLift, and CS.
Note that in the CS scenario, the simulation is given the complete output $\yy$ instead of its squared magnitudes.}
\label{fig:nVSk}
\end{figure}

With the same simulation setup, we compare the accuracy of CPRL with the PhaseLift approach and the CS approach in Figure~\ref{fig:nVSk}.
First, note that CS is not applicable to phase retrieval problems in practice, since it assumes the phase of the observation is also given.
Nevertheless, the simulation shows CPRL via the SDP solution only requires a slightly higher sampling rate to achieve the same success rate
as CS, even when the phase of the output is missing. Second, similar to the discussion in Example \ref{ex:ex1}, without enforcing the sparsity constraint in \eqref{eq:phase-lift},
PhaseLift would fail to recover correct sparse signals in the low sampling rate regime.

It is also interesting to see the performance as $n$ and $N$ vary and $k$ held fixed. We therefore use the same setup as
in Figure~\ref{fig:nVSk} but now  fixed $k=2$ and for $n=10,\dots,60,$
gradually increased $N$ until CPRL recovered the true sparsity pattern
with 95\% success rate. The same procedure is repeated to
evaluate PhaseLift and CS. The results are shown in Figure \ref{fig:n-vs-N}.

\begin{figure}[th!]
\centering
\includegraphics[width = 0.8\columnwidth]{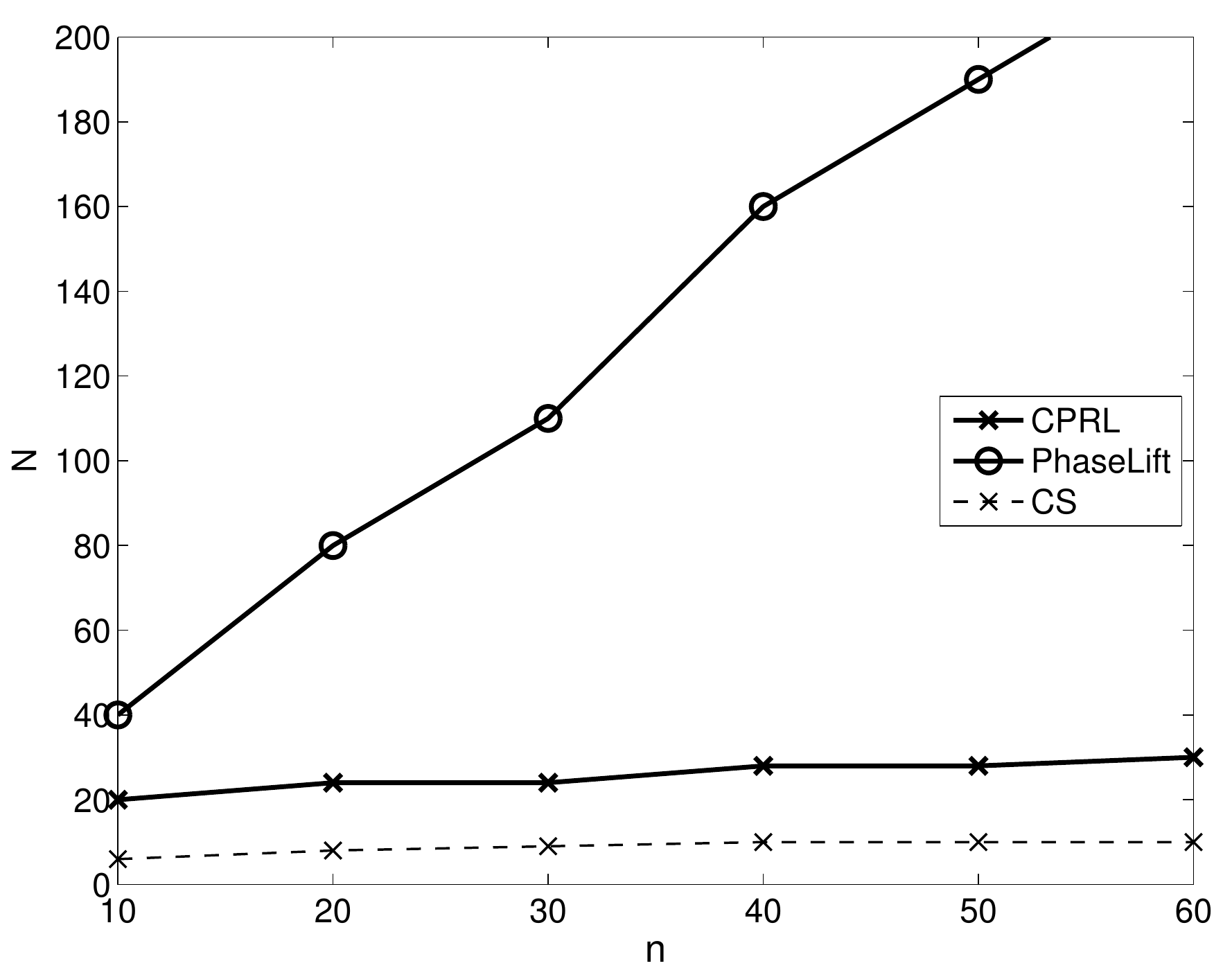}
\caption{The curves of 95\% success rate for CPRL, PhaseLift, and CS.
Note that the CS simulation is given the complete output $\yy$ instead of its squared magnitudes.} \label{fig:n-vs-N}
\end{figure}

Compared to Figure \ref{fig:nVSk}, we can see that the degradation from CS to CPRL when the phase information is omitted is largely affected by the sparsity of the signal. More specifically, when the sparsity $k$ is fixed, even when the dimension $n$ of the signal increases dramatically, the number of squared observations to achieve accurate recovery does not increase significantly for both CS and CPRL.

Next, we calculate the quantity $\frac{1}{2} \left ( 1 +  \frac{1}{\mu(\B)}\right)$, as Theorem \ref{thm:guartee2} shows that when
\begin{equation}
\|\bar X\|_0 < \frac{1}{2} \left ( 1 + \frac{1}{\mu(\B)}\right)
\end{equation}
and $\bar X$ has rank 1, then $X^* = \bar X$. The quantity is plotted for a number
of different $N$ and $n$'s in Figure \ref{fig:coher}. From the plot it
can be concluded that if the solution $\bar X$ has rank 1 and
only a single nonzero component for a choice of $2000 \geq n \geq
10,\, 45 \geq  N\geq 5$, Theorem \ref{thm:guartee2} can guarantee that $\bar X=X^*$.
We also observer that Theorem \ref{thm:guartee2} is pretty conservative,
since from Figure~\ref{fig:n-vs-N}
we have that with high probability we needed $N > 25$ to guarantee that
a two sparse vector is recovered correctly for $n=20$.

\begin{figure}[th!]
\centering
\includegraphics[width = 0.8\columnwidth]{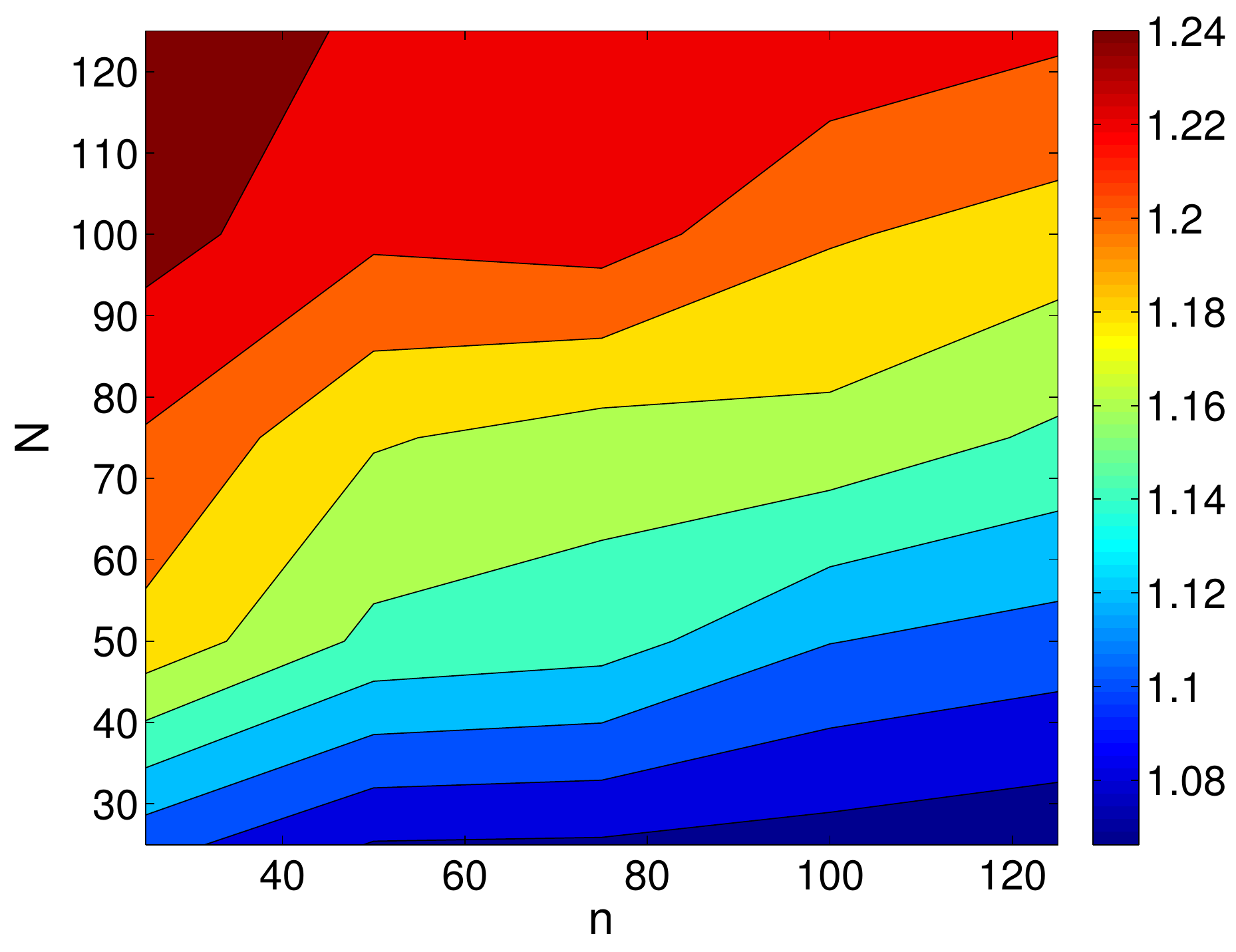}
\caption{A contour plot of the quantity $\frac{1}{2} \left ( 1 +
  \frac{1}{\mu(\B)}\right)$. $\mu$ is taken as the average over
10 realizations of $B$.} \label{fig:coher}
\end{figure}

\subsection{Audio Signals}

In this section, we further demonstrate the performance of CPRL using signals from a real-world audio recording. The timbre of a particular note on an instrument is determined by the fundamental frequency, and several overtones. In a Fourier basis, such a signal is sparse, being the summation of a few sine waves. Using the recording of a single note on an instrument will give us a naturally sparse signal, as opposed to synthesized sparse signals in the previous sections. Also, this experiment will let us analyze how robust our algorithm is in practical situations, where effects like room ambience might color our otherwise exactly sparse signal with noise.

Our recording $\zz\in\Re^s$ is a real signal, which is assumed to be sparse in a Fourier basis. That is, for some sparse $\xx \in \Ce^n$, we have $\zz = F_{inv} \xx$, where $F_{inv} \in \Ce^{s \times n}$ is a matrix representing a transform from Fourier coefficients into the time domain. Then, we have a randomly generated mixing matrix with normalized rows, $R \in \Re^{N \times s}$, with which our measurements are sampled in the time domain: 
\begin{equation}
\yy = R \zz = RF_{inv} \xx.
\label{eq:audio-sampling}
\end{equation}
Finally, we are only given the magnitudes of our measurements, such that $\bb = |\yy|^2 = |R \zz|^2$. 

For our experiment, we choose a signal with $s = 32$ samples, $N = 30$ measurements, and it is represented with $n = 2s$ (overcomplete) Fourier coefficients. Also, to generate $F_{inv}$, the $\Ce^{n \times n}$ matrix representing the Fourier transform is generated, and $s$ rows from this matrix are randomly chosen.

The experiment uses part of an audio file recording the sound of a tenor saxophone. The signal is cropped so that the signal only consists of a single sustained note, without silence. Using CPRL to recover the original audio signal given $\bb$, $R$, and $F_{inv}$, the algorithm gives us a sparse estimate $\xx$, which allows us to calculate $\zz_{est} = F_{inv} \xx$. We observe that all the elements of $\zz_{est}$ have phases that are $\pi$ apart, allowing for one global rotation to make $\zz_{est}$ purely real. This matches our previous statements that CPRL will allow us to retrieve the signal up to a global phase.

We also find that the algorithm is able to achieve results that
capture the trend of the signal using less than $s$ measurements. In
order to fully exploit the benefits of CPRL that allow us to achieve
more precise estimates with smaller errors using fewer measurements
relative to $s$, the problem should be formulated in a much higher
ambient dimension. However, using the CVX Matlab toolbox by
\cite{CVX1}, we already ran into computational and memory limitations
with the current implementation of the CPRL algorithm. These results
highlight the need for a more efficient numerical implementation of
CPRL as an SDP problem.

\begin{figure}[th!]
\centering
\includegraphics[width = 0.8\columnwidth]{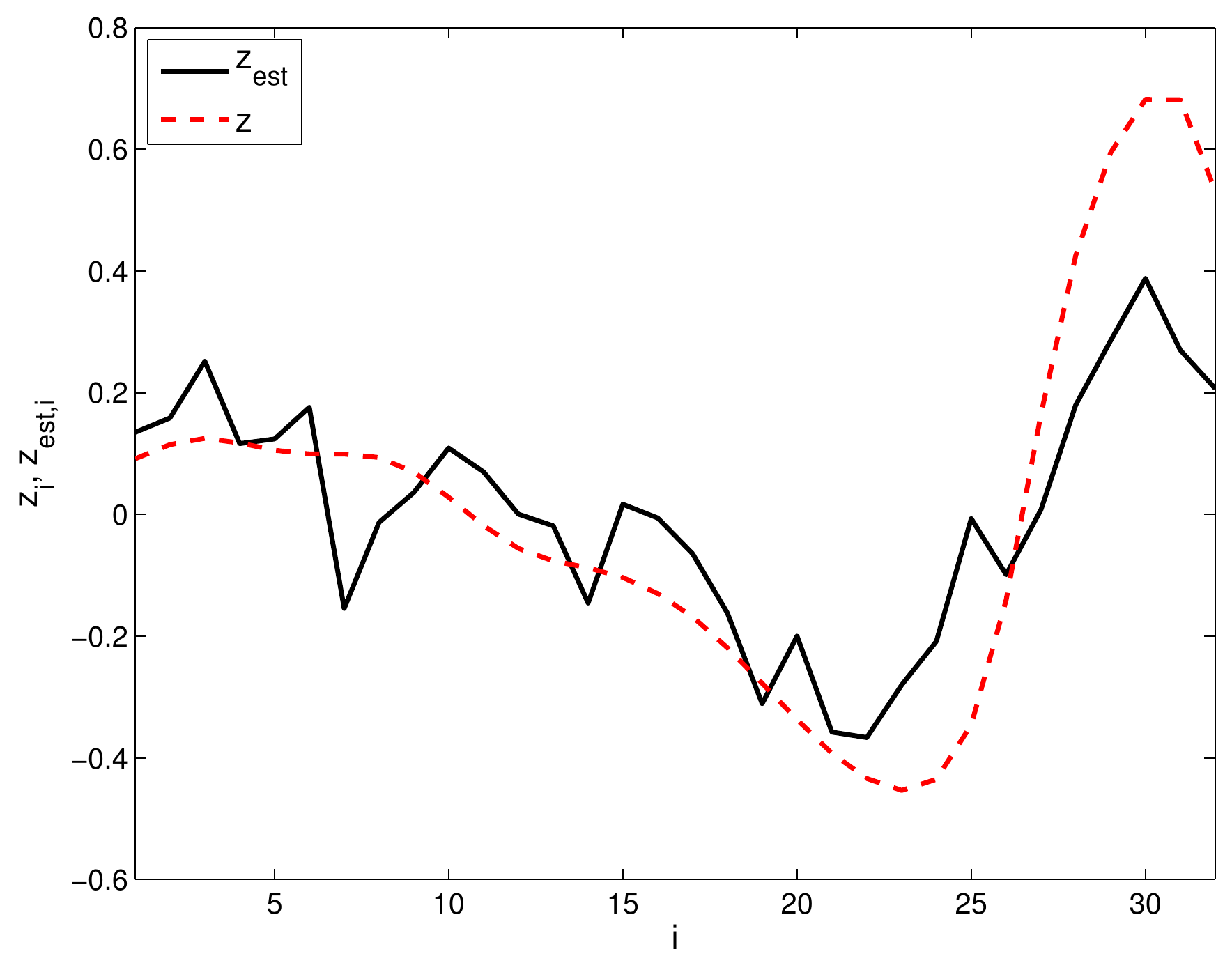}
\caption{The retrieved signal $\zz_{est}$ using CPRL versus the original audio signal $\zz$.}
\label{fig:experiment_figure_1}
\end{figure}

\begin{figure}[th!]
\centering
\includegraphics[width = 0.8\columnwidth]{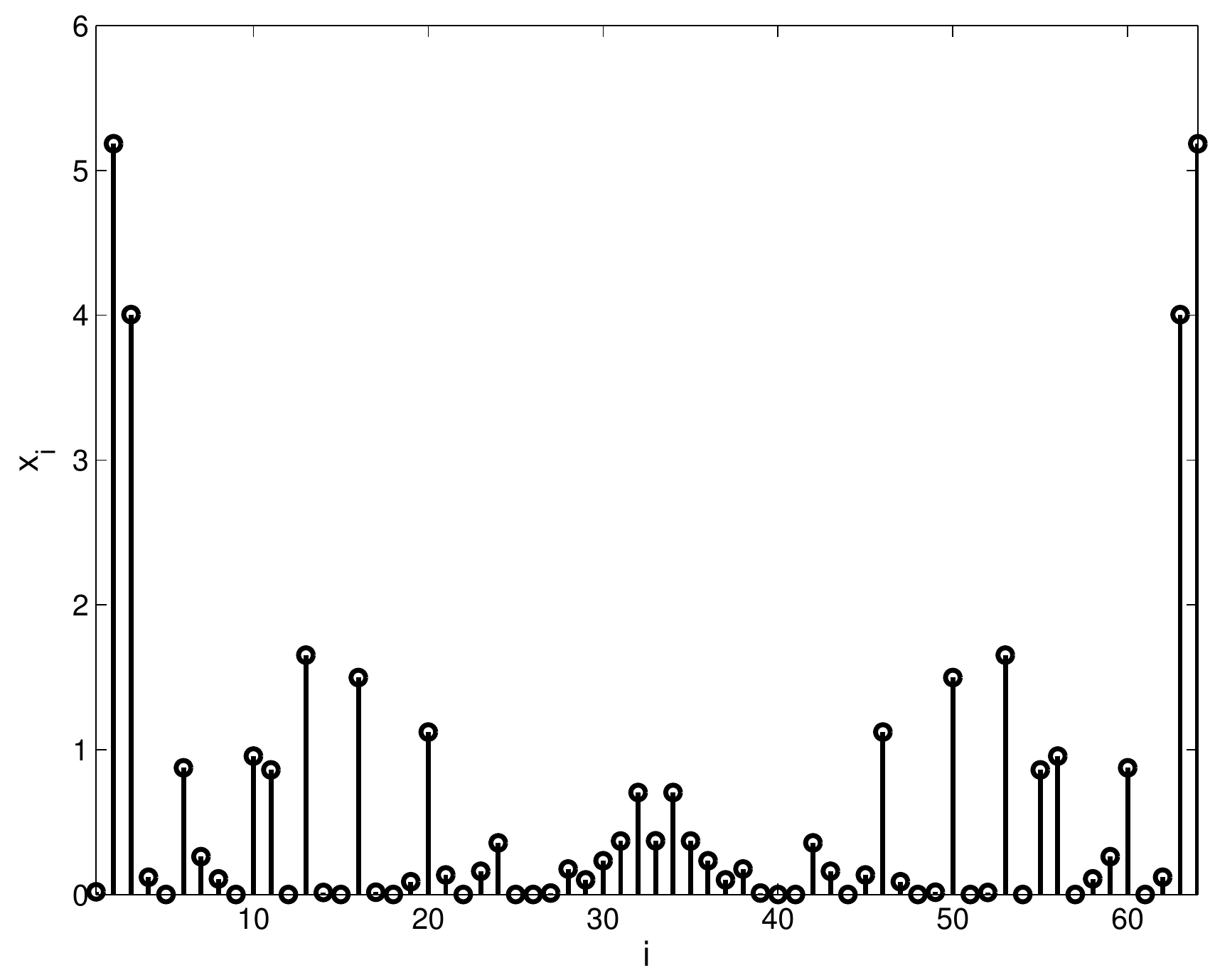}
\caption{The magnitude of $\xx$ retrieved using CPRL. The audio signal $\zz_{est}$ is obtained by $\zz_{est} = F_{inv} \xx$.}
\label{fig:experiment_figure_2}
\end{figure}

\section{Conclusion and Discussion}
A novel method for the compressive phase retrieval problem has been presented. The method
takes the form of an SDP problem and provides the
means to use compressive sensing in applications where only squared
magnitude measurements are available. The convex formulation gives it
an edge over previous presented approaches and numerical illustrations
show state of the art performance.

One of the future directions is improving the speed of the standard SDP solver, i.e., interior-point methods, currently used for the CPRL algorithm. The authors have previously introduced efficient numerical acceleration techniques for $\ell_1$-min \citep{YangA2010-ICIP} and Sparse PCA \citep{NaikalN2011-ICCV} problems. We believe similar techniques also apply to CPR. Such accelerated CPR solvers would facilitate exploring a broad range of high-dimensional CPR applications in optics, medical imaging, and computer vision, just to name a few.

\section{Acknowledgement}
Ohlsson is partially supported by the Swedish foundation for strategic
 research in the center MOVIII, the Swedish Research Council in
 the Linnaeus center CADICS, the European Research Council under the advanced
 grant LEARN, contract 267381, and a postdoctoral grant from the
 Sweden-America Foundation, donated by ASEA's Fellowship Fund.
 Sastry and Yang are partially supported by an ARO MURI grant W911NF-06-1-0076.
 Dong is supported by the NSF Graduate Research Fellowship under grant DGE 1106400, and by the Team for Research in Ubiquitous Secure Technology (TRUST), which receives support from NSF (award number CCF-0424422).

\bibliography{refHO}
\end{document}